\documentclass[11pt, oneside]{article}   	% use "amsart" instead of "article" for AMSLaTeX format
\usepackage{geometry}                		% See geometry.pdf to learn the layout options. There are lots.
\usepackage{amssymb}
\usepackage{amsmath}
\usepackage{amsthm}

\usepackage{graphicx}

\newtheorem{theorem}{Theorem}[section]
\newtheorem{lemma}[theorem]{Lemma}
\newtheorem{proposition}[theorem]{Proposition}
\newtheorem{corollary}[theorem]{Corollary}

\theoremstyle{definition}

\theoremstyle{remark}

\numberwithin{equation}{section}

\newcommand{\N}{\mathbb{N}}
\newcommand{\R}{\mathbb{R}}

\usepackage{xcolor}

\title{One-Sided Derivative of Distance to a Compact Set}

\author{Logan S. Fox\footnote{Fariborz Maseeh Dept. of Math. and Stat., Portland State University; contact: logfox@pdx.edu}, Peter Oberly\footnote{Department of Mathematics, Oregon State University; contact: oberlyp@oregonstate.edu}, J.J.P. Veerman\footnote{Fariborz Maseeh Dept. of Math. and Stat., Portland State University; contact: veerman@pdx.edu} }

\date{\today}

\begin{document}

\maketitle

\begin{abstract}
We give a complete and self-contained proof of a folklore theorem which says that in an Alexandrov
space the distance between a point $\gamma(t)$ on a geodesic $\gamma$ and a compact set $K$ is a right-differentiable function of $t$. Moreover, the value of this right-derivative is given by 
the negative cosine of the minimal angle between the geodesic and any shortest path to the 
compact set (Theorem \ref{thm:compact set derivative}). 
Our treatment serves as a general introduction to metric geometry and relies only on the basic elements, such as comparison triangles and upper angles.
\end{abstract}

\section{Introduction} % INTRODUCTION %

Let \((X,d)\) be a metric space. Given a compact set \(K \subseteq X\) and a geodesic \(\gamma:[0,T]\to X\), the distance from \(\gamma\) to \(K\) at any given time is defined by the function 
\[\ell(t) = d(\gamma(t),K) . \] %This article considers the right-derivative of the function \(\ell\). 
In an Alexandrov space (see Section \ref{alexandrov spaces} for definition), if we replace the compact set with a point, \(K = \{p\}\), it is well known that 
\begin{equation}\label{limit}
\lim_{t\to 0^+} \frac{\ell(t) -\ell(0)}{t} = -\cos(\angle_{\min}) 
\end{equation} 
where \(\angle_{\min}\) is the infimum of angles between \(\gamma\) and any distance minimizing path connecting \(\gamma(0)\) to \(p\). This result is commonly known as the First Variation Formula (after the similar result for Riemannian manifolds) and can be found in \cite[Proposition 3.3]{abn86}, \cite[Corollary II.3.6]{bridson}, \cite[Corollary 4.5.7]{bbi2001}, and \cite[Corollary 62]{plaut}.

It is asserted, in publications such as \cite[Exercise 4.5.11]{bbi2001} and \cite[Example 11.4]{bgp92}, that the first variation formula (\ref{limit}) still holds for the distance to an arbitrary compact set \(K\), with \(\angle_{\min}\) representing the infimum of angles between \(\gamma\) and any distance minimizing path connecting \(\gamma(0)\) to \(K\). 
However, neither of these sources (\cite{bbi2001} and \cite{bgp92}) provide a proof. % or state the full result (that is, it holds for curvature bounded above or below by an arbitrary \(k\)). 
On the other hand, there is a proof of the first variation formula in \cite[Proposition 9.4]{lytchak}, which is further generalized to hold in a class of geometric metric spaces broader than Alexandrov spaces. This generalization, however, is achieved at the cost of some very technical machinery. 
Namely, it involves constructing a tangent space at every point \(p\) via the ultralimit of blow-ups of the pointed space \((X,p)\). 
Further, the angle between geodesics may not be well defined in this context, so Lytchak is required to make use of Busemann functions to form a metric on this tangent space, much like the law of cosines in the Euclidean case. 
In this more general formulation, the right-hand side of (\ref{limit}) is replaced by a Busemann function. 

Our goal is to present a complete and self-contained proof of (\ref{limit}), which relies solely on the fundamentals of metric geometry. As such, this article may also serve as a gentle yet rigorous introduction to the theory of Alexandrov spaces.
Our approach is largely based on techniques presented in \cite{bbi2001}, with insights taken from \cite{bridson}, \cite{plaut}, \cite{shiohama}, and others.

\subsection{Background and Motivation} % BACKGROUND %

Loosely speaking, an Alexandrov space is metric space which satisfies enough structural requirements for some classic geometric notions such as geodesics, angles, and curvature make sense. In this context, curvature is based on a local bound, which is obtained through comparison to one of the two-dimensional space forms - hyperbolic, spherical, or Euclidean space. 

The importance of Alexandrov spaces can be seen through a few examples. First and foremost, all Riemannian manifolds are in fact Alexandrov spaces. %when viewed through the lens of synthetic geometry. 
On the other hand, limits of Riemannian manifolds (in the Gromov-Hausdorff metric) with lower-bounded curvature and upper-bounded diameter are not necessarily Riemannian manifolds, but are always Alexandrov spaces (this is Gromov's compactness theorem). A simple example is the surface of an \(n\)-dimensional cube in \(\R^{n+1}\). The surface of a cube is clearly not a smooth manifold, as it has sharp corners and edges; however, it is an Alexandrov space and can be attained as the limit as \(q\) tends to infinity of the smooth spherical \(n\)-manifolds \(\{x\in \R^{n+1} : |x_1|^q + \cdots + |x_{n+1}|^q = 1\}\), each of which has non-negative curvature.

Another reason that Alexandrov spaces are significant, is that certain general characteristics or properties that we generally associate with smoothness - or smooth structures - in fact hold in a more general context. This is somewhat akin to observing that at a local extremum of a smooth function from \(\R\) to itself, this function admits a horizontal tangent line. With a careful definition of tangent line, this is true in a more general setting. Similarly, in this case we have that (\ref{limit}) is not only true for Riemannian manifolds (with their smooth metric), but also in a broader set of spaces. 
Studying geometry in a `weaker' setting gives insight into both the properties in question, and the smooth structures themselves.

The first variation formula, in particular, is a fundamental property of distances in Alexandrov space, and has found numerous applications in geometry. We name a few applications here. In \cite{shiohama1996cut}, the authors use the first variation formula to study cut-loci on spheres in Alexandrov space. 
In the  paper \cite{perelman1994quasigeodesics}, the authors use this theorem to prove that the length of convex curves is preserved when taking limits of Alexandrov spaces under suitable conditions. 
Another consequence of note is that the first variation formula allows one to introduce a metric on the space of directions emanating from a point to a compact set, which is fundamental in studying tangent cones (see \cite{busemann1970recent}). 
Finally, in \cite{petrunin} a version of the first variation formula %for distance to extremal subsets of Alexandrov space 
is used to prove a deep glueing theorem for Alexandrov spaces with boundary (informally: if \(X\) and \(Y\) are Alexandrov spaces with boundary and of curvature \(\geq k\), then gluing \(X\) and \(Y\) along their boundaries produces an Alexandrov space of curvature \(\geq k\)).

Our own motivation is twofold. It is well known (\cite{Veer} and references therein) that if $K$ is a
convex set in $\R^n$, then the derivative of $\ell(t)$ equals the negative cosine of the angle between \(\gamma\) and the distance minimizing path connecting \(\gamma(0)\) to \(K\). The result considered here is a strong generalization of that fact. 
Additionally, the result is crucial for the study of mediatrices; that is, for fixed points \(p,q\in X\), the set \(\{x\in X : d(x,p) = d(x,q)\}\). In \cite{HPV}, the first variation formula (for distance to a point) is used to show that mediatrices on compact Riemannian surfaces have a Lipschitz structure. There is a natural generalization of mediatrices as the equidistant set between disjoint compact subsets, and the result we prove here is a necessary step in extending \cite{HPV} to this case.

\section{Fundamentals of Metric Geometry} %%% FUNDAMENTALS %%%

This section begins with two of the most basic notions of metric geometry: length and comparison configurations; followed by a survey of (upper) angles between geodesics. Although the use of angles has somewhat faded in the modern theory (for generalizations beyond Alexandrov spaces), it can aid geometric intuition, especially to the beginner. The definitions and results presented in this section are well-established, and can also be found in the introductory chapters of \cite{aleksandrovzalgaller}, \cite{bbi2001}, \cite{reshetnyak}, and \cite{bridson}.

\subsection{Length Spaces} %%% LENGTH SPACES %%% 

In order to establish a synthetic geometry in a metric space, we rely on paths in the space to get from one point to another. To that end, the metric needs to align with our intuitive idea of how distance is measured; the distance between any two points is the length of a `straight line' connecting them. Here we build the vocabulary and structure for these kinds of metric spaces, known as length spaces. 

Let \((X,d)\) be a metric space. A path (or curve) in \(X\) is a continuous injective function \(\gamma: [a,b] \to X\) where \([a,b]\) is an interval of \(\R\) (possibly degenerate). We define the length of any path \(\gamma\) as the supremum of the distance along finite partitions of the path:
\[ L(\gamma) = \sup \bigg\{ \sum_{k=1}^{n-1} d\big( \gamma(t_{k}), \gamma(t_{k+1}) \big) : a=t_1 < t_2<\cdots < t_n = b \bigg\} . \]
If for all \(x,y\in X\),
\[ d(x,y) = \inf \{ L(\gamma) : \gamma \text{ is a path connecting } x \text{ and } y \} \]
then the metric \(d\) is said to be intrinsic. A path-connected metric space with an intrinsic metric is known as a \emph{length space}.

Two paths, \(\gamma:[0,T] \to X\) and \(\eta: [0,S] \to X\), which have the same image but are not the same function are said to have different \emph{parameterizations}. A path \(\gamma: [0,T] \to X\) is parameterized by arc-length (or \emph{unit-speed}, for short) if for any \(t,t' \in [0,T]\),
\[ L(\gamma|_{[t,t']}) = |t' - t| . \]
A sequence of paths \(\{\gamma_n\}_{n=1}^\infty\) is said to converge uniformly to a path \(\gamma\) if each \(\gamma_n\) admits a parameterization such that \(\{\gamma_n\}_{n=1}^\infty\) converges uniformly to some parameterization of \(\gamma\).

Finally, a \emph{shortest path} is a unit-speed curve \(\gamma: [0,T] \to X\) such that the length of \(\gamma\) is precisely the distance between its endpoints; \(L(\gamma) = d(\gamma(0),\gamma(T)) = T\). Any curve which is locally a shortest path is known as a \emph{geodesic}. 
In a length space which is both complete and locally compact, we make use of the Hopf-Rinow Theorem, although we will not reference it directly.

\begin{theorem}[Hopf-Rinow]
If \(X\) is a complete and locally compact length space, then every closed and bounded subset of \(X\) is compact; and any two points in \(X\) can be connected by a shortest path.
\end{theorem}

\subsection{Comparison Triangles} %%% Comparison Triangles %%%

We denote by \(M^2_k\) the \(2\)-dimensional simply-connected space form\footnote{A space form is a complete Riemannian manifold of constant sectional curvature.} of curvature \(k\), equipped with intrinsic metric \(d_k\) induced by the Riemannian metric. The diameter of the space \(M^2_k\) is denoted \(D_k\) and defined by
\[
D_k = \sup \{ d_k(x,y) : x,y \in M^2_k\} = \begin{cases} \pi/\sqrt{k} & \text{ for } k > 0 \\ \infty & \text{ for } k\leq 0 .\end{cases}
\]
Given any three points \(x,y,z\in X\) with \(d(x,y) + d(y,z) + d(x,z) < D_k\), we can fix three points \(\bar{x}\), \(\bar{y},\) and \(\bar{z}\) in \(M^2_k\) such that
\begin{equation}\label{triangle sides}
d(x,y) = d_k( \bar{x} , \bar{y} ) , \ d(x,z) = d_k( \bar{x} , \bar{z} ) , \ \text{and } d(y,z) = d_k( \bar{y},\bar{z} ) .
\end{equation}
The points \(\bar{x}\), \(\bar{y}\), and \(\bar{z}\), together with the shortest paths joining them,
form a geodesic triangle in \(M^2_k\), which we call the \emph{comparison triangle} and denote it
\(\overline{\Delta}(x,y,z)\). Such a comparison triangle is unique up to isometry. The interior angle
in the geodesic triangle \(\overline{\Delta}(x,y,z)\) (in \(M^2_k\)) with vertex \(\bar{x}\) is denoted
\(\angle^k_x(y,z)\)
and referred to as the \emph{\(k\)-comparison angle}.

\subsection{Upper Angles}

If \(\gamma:[0,T] \to X\) and \(\eta:[0,S]\to X\) are shortest paths in \(X\) with \(\gamma(0) = \eta(0)\), then for any sufficiently small\footnote{Sufficiently small meaning the inequality \(d(\gamma(0),\gamma(t)) + d(\gamma(0),\eta(s)) + d(\gamma(t),\eta(s)) < 2D_k\) is satisfied. } \(t\in (0,T]\) and \(s\in (0,S]\), we can consider the comparison triangle \(\overline{\Delta}\big(\gamma(0),\gamma(t),\eta(s)\big)\). The \emph{upper angle} between \(\gamma\) and \(\eta\) is defined as
\begin{align*} 
\angle^+_{\gamma(0)}(\gamma,\eta) & = \limsup_{t,s\to0^+} \angle^k_{\gamma(0)} (\gamma(t),\eta(s)) 
\\ & = \lim_{\varepsilon\to 0} \sup \big\{ \angle^k_{\gamma(0)} (\gamma(t),\eta(s)) : 0<s,t\leq \varepsilon \big\} . 
\end{align*}
When it is understood that \(\gamma(0)\) is the point at which we are measuring the angle, the subscript for the vertex is often omitted (i.e. \(\angle^+_{\gamma(0)}(\gamma,\eta) = \angle^+(\gamma,\eta)\) and \(\angle^k_{\gamma(0)}(\gamma(t),\eta(s)) = \angle^k(\gamma(t),\eta(s))\)).

Besides the upper angle, one may also consider the \emph{lower angle} between two shortest paths, which is defined as \(\angle^-(\gamma,\eta) = \liminf_{s,t\to 0^+} \angle^k (\gamma(t),\eta(s))\). If the upper angle and lower angle are equal, then we say the angle exists and denote it by \(\angle(\gamma, \eta)\).
We note that the upper and lower angles are indeed independent of the curvature of the space form chosen (as per \cite{plaut}, all space forms are infinitesimally Euclidean; see also Appendix \ref{sec: inf-eucl}).

The following proposition is commonly referred to as the triangle inequality for angles. The proof given here the same as that found in \cite{bridson}. 

\begin{proposition}\label{triangle inequality angles}
Let \(X\) be a length space and let \(\gamma\), \(\eta\), and \(\sigma\) be shortest paths in \(X\) with \(\gamma(0)=\eta(0)=\sigma(0)\). Then \(\angle^+ (\gamma,\eta) \leq \angle^+(\gamma,\sigma)+\angle^+(\sigma,\eta)\).
\end{proposition}

\begin{proof}
If \(\angle^+(\gamma,\sigma) + \angle^+(\sigma,\eta)\geq\pi\), then the result is trivial, so we assume that \(\angle^+(\gamma,\sigma) + \angle^+(\sigma,\eta)<\pi\). By way of contradiction, suppose that there is an \(\varepsilon>0\) such that
\begin{equation}\label{eq: triangle contradiction}
\angle^+(\gamma,\eta) > \angle^+(\gamma,\sigma) + \angle^+(\sigma,\eta) + \varepsilon .
\end{equation}
By the definition of \(\limsup\) there is a \(\delta>0\) such that
\begin{align}
\angle^k(\gamma(t),\eta(r)) & > \angle^+(\gamma,\eta) -\varepsilon/3 \quad \text{for some } t,r<\delta \label{eq: triangle 1}
\\ \angle^k(\gamma(t),\sigma(s)) & < \angle^+(\gamma,\sigma) +\varepsilon/3 \quad \text{for all } t,s < \delta \label{eq: triangle 2}
\\ \angle^k(\sigma(s),\eta(r)) & < \angle^+(\sigma,\eta)+\varepsilon/3 \quad \text{for all } s,r<\delta . \label{eq: triangle 3}
\end{align}
Fix \(t\) and \(r\) satisfying (\ref{eq: triangle 1}) and let \(\bar{p},\bar{t},\bar{r}\in M^2_k\) be such that
\(t = d_k(\bar{t},\bar{p})\), \(r = d_k(\bar{r},\bar{p})\), and
\[\angle^k(\gamma(t),\eta(r)) > \theta_{\bar{t},\bar{r}} > \angle^+(\gamma,\eta) -\varepsilon/3 \]
where \(\theta_{\bar{t},\bar{r}}\) is the angle between \(\bar{p}\bar{t}\) and \(\bar{p}\bar{r}\) in \(M^2_k\). The left side of he above ineqality tells us that \(d(\gamma(t),\eta(r)) > d_k(\bar{t},\bar{r})\). Combining the right side of the above inequality with (\ref{eq: triangle contradiction}), we have
\[ \theta_{\bar{t},\bar{r}} > \angle^+(\gamma,\sigma) + \angle^+(\sigma,\eta) + 2\varepsilon/3 . \]
Therefore, we can fix \(\bar{s}\in M^2_k\) along the path \(\bar{t}\bar{r}\) such that
\[ \theta_{\bar{t},\bar{s}} > \angle(\gamma,\sigma)+\varepsilon/3 \quad \text{and} \quad \theta_{\bar{s},\bar{r}} > \angle(\sigma,\eta)+\varepsilon/3 . \]
Set \(s = d_k(\bar{s},\bar{p})\). Since \(d_k(\bar{s},\bar{p})\leq \max\{d_k(\bar{t},\bar{p}),d_k(\bar{r},\bar{p})\} < \delta\), by (\ref{eq: triangle 2}) and (\ref{eq: triangle 3}) we have
\[ \theta_{\bar{t},\bar{s}} > \angle^k(\gamma(t),\sigma(s)) \quad \text{and} \quad \theta_{\bar{s},\bar{r}} > \angle^k(\sigma(s),\eta(r)) . \]
It follows that \(d_k(\bar{t},\bar{s}) > d(\gamma(t),\sigma(s))\) and \(d_k(\bar{s},\bar{r})>d(\sigma(s),\eta(r))\).
Thus, we have
\[d(\gamma(t),\eta(r)) > d_k(\bar{t},\bar{r}) = d_k(\bar{t},\bar{s})+d_k(\bar{s},\bar{r}) > d(\gamma(t),\sigma(s)) + d(\sigma(s),\eta(r)) \]
which contradicts the triangle inequality. 
\end{proof}

\subsection{Two Results for Thin Triangles} % THIN TRIANGLES %

The next lemma is arguably the crux of this work. As observed above, it is clear that small
triangles in space forms are essentially Euclidean. However, what we need here are the properties of long, thin triangles, that is: triangles with only one small side (and two long sides). The
surprising --- and perhaps counter-intuitive --- fact is that these also behave like Euclidean triangles!

\begin{lemma}\label{lem: cosine inequality}
Let \(X\) be a length space and let \(\gamma:[0,T]\to X\) and \(\eta: [0,S] \to X\) be shortest paths such that \(\gamma(0) = \eta(0)\). Then for fixed \(s\) such that \(0<s<D_k\), 
\[
\lim_{t\to 0^+} \left| \cos\Big(\angle^k (\gamma(t),\eta(s))\Big) - \dfrac{s - d(\gamma(t),\eta(s))}{t}\right| = 0.
\]
\end{lemma}

\begin{proof}
We first look at \(k=0\) and summarize the proof found in \cite[Lemma 4.5.5]{bbi2001}.
For simplicity of notation, let \(\theta = \angle^k (\gamma(t),\eta(s))\) and \(d = d(\gamma(t),\eta(s))\). Recall that \(t=d(\gamma(0),\gamma(t))\) and \(s=d(\gamma(0),\eta(s))\). Employing the Euclidean law of cosines, we find
\[
d^2=s^2+t^2-2st\cos ( \theta ) .
\]
A trivial computation confirms that
\[
\left| \cos ( \theta ) - \dfrac{s-d}{t}\right|=\left| \dfrac{s-d}{t}\cdot \dfrac{d-s}{2s} +
\dfrac{t}{2s}\right|.
\]
By the triangle inequality, $|s-d|=|d-s| \leq t$, which gives the desired result. 

We next consider the case \(k>0\). If we radially project
the triangle with sides of lengths $s$, $t$, and $d$ to the unit-sphere, we can use
the spherical law of cosines to derive
\begin{align*} 
\cos(\theta) & = \frac{\cos(d\sqrt{k}) - \cos(t\sqrt{k}) \cos(s\sqrt{k})}{\sin(t\sqrt{k})\sin(s\sqrt{k})} 
\\ & = \frac{\cos(d\sqrt{k}) - \cos(s\sqrt{k})}{\sin(t\sqrt{k})\sin(s\sqrt{k})} + \frac{\cos(s\sqrt{k}) (1 - \cos(t\sqrt{k}))}{\sin(t\sqrt{k})\sin(s\sqrt{k})} . 
\end{align*} 
Recall from the trigonometric relations that 
\begin{align*} 
\cos(d\sqrt{k})-\cos(s\sqrt{k}) & = 2\sin\bigg(\frac{(s+d)\sqrt{k}}{2} \bigg)  \sin\bigg(\frac{(s-d)\sqrt{k}}{2}\bigg) , \\ 
1 - \cos(t\sqrt{k}) & = 2\sin^2\bigg( \frac{t\sqrt{k}}{2} \bigg) , \\ 
\text{ and } \sin(t\sqrt{k}) & = 2\sin\bigg( \frac{t\sqrt{k}}{2} \bigg) \cos\bigg( \frac{t\sqrt{k}}{2} \bigg) . 
\end{align*}
Combining all of the above, we get 
\[ \cos(\theta) = \bigg( \frac{\sin\big(\frac{(s+d)\sqrt{k}}{2} \big)}{\sin(s\sqrt{k})} \bigg) \bigg( \frac{2\sin\big(\frac{(s-d)\sqrt{k}}{2}\big)}{\sin(t\sqrt{k})} \bigg) + \frac{\cos(s\sqrt{k})\sin\big( \frac{t\sqrt{k}}{2} \big)}{\sin(s\sqrt{k})\cos\big( \frac{t\sqrt{k}}{2} \big)}. \] 
Note that \(d \to s\) as \(t\to 0\). Using the limit of $\frac{\sin x}{x}$, we find 
\begin{align*} 
& \lim_{t\to 0^+} \frac{\sin\big(\frac{(s+d)\sqrt{k}}{2} \big)}{\sin(s\sqrt{k})} = 1 , \\ 
& \lim_{t\to 0^+} \frac{2\sin\big(\frac{(s-d)\sqrt{k}}{2}\big)}{\sin(t\sqrt{k})} = \lim_{t\to 0^+} \frac{s-d}{t} , \\ 
\text{and } & \lim_{t\to 0^+} \frac{\cos(s\sqrt{k})\sin\big( \frac{t\sqrt{k}}{2} \big)}{\sin(s\sqrt{k})\cos\big( \frac{t\sqrt{k}}{2} \big)} = 0
\end{align*} 
which gives us \( \lim_{t\to 0^+} \left| \cos (\theta) - \frac{s-d}{t} \right| = 0\). 
%As \(0<s<D_k / 2\), taking \(C = \frac{\varepsilon s}{t} + \frac{\varepsilon s \sqrt{k} \cos (s\sqrt{k})}{2\sin (s\sqrt{k})}\) gives us the result. 

The proof for \(k<0\) follows from the relationships \(\cos(ix) = \cosh(x)\) and \(\sin(ix) = i\sinh(x)\) and is very similar. It can be found in \cite[p. 11]{alexandrov51} and \cite[Lemma 4.1]{shiohama}. 
\end{proof}

%The upshot is that
%\[ \left| \cos (\theta) - \dfrac{s-d}{t} \right| \leq \varepsilon + \frac{\varepsilon s\sqrt{|k|} \cosh \big( s\sqrt{|k|}\big)}{2\sinh \big(s\sqrt{|k|}\big)} \bigg(\dfrac{t}{s}\bigg). \]
%In this case, taking \(C = \frac{\varepsilon s}{t} + \frac{\varepsilon s\sqrt{|k|} \cosh \big( s\sqrt{|k|}\big)}{2\sinh \big(s\sqrt{|k|}\big)}\) finishes the proof.

\begin{lemma}
If \(X\) is a length space, then for all shortest paths \(\gamma:[0,T]\to X\) and \(\eta:[0,S]\to X\) with \(\gamma(0)=\eta(0)\), for every fixed \(s>0\), we have
\[\limsup_{t\to 0^+} \angle^k (\gamma(t),\eta(s)) \leq \angle^+ (\gamma,\eta) .
\]
\label{lem:strong angle}
\end{lemma}

\begin{proof} If \(s' < s\), then by the triangle inequality
\[ s-s' \geq d \big( \gamma(t),\eta(s) \big) - d \big( \gamma(t),\eta(s') \big) \]
which gives us 
\[ s - d \big( \gamma(t),\eta(s) \big) \geq s' - d \big( \gamma(t),\eta(s') \big) . \]
Substituting this into Lemma \ref{lem: cosine inequality}, we see that
\[ \liminf_{t\to 0^+} \cos \big( \angle^k (\gamma(t),\eta(s)) \big) \geq  \liminf_{t\to 0^+} \cos \big( \angle^k (\gamma(t),\eta(s')) \big) . \]
As cosine is nonincreasing on \([0,\pi]\), we have
\[ \cos \bigg( \limsup_{t\to0^+} \angle^k (\gamma(t),\eta(s)) \bigg) \geq \cos\bigg( \limsup_{s,t \to 0^+} \angle^k (\gamma(t),\eta(s)) \bigg) . \]
The right hand equals $\cos ( \angle^+(\gamma,\eta) )$.
\end{proof}

\section{Alexandrov Spaces}\label{alexandrov spaces} %%% ALexaNDROV SPACES %%%

Here we give an account of metric spaces which exhibit bounded curvature in the sense of Alexandrov, which most nearly resemble Riemannian maifolds with bounded sectional curvature. The idea of bounded curvature in a metric space is certainly not unique to Alexandrov spaces; another common example is that developed by Busemann for non-positive curvature (see for example \cite{papadopoulos}). Even recently, new characterizations for spaces of bounded curvature, such as in \cite{jost1} and \cite{jost2}, have given rise to further possibilities in the application of metric geometry.

\subsection{Bounded Curvature}

A length space \(X\) is said to be of \emph{curvature bounded above} (or \emph{curvature \(\leq k\)}) if there is a \(k\in \R\) for which the following holds: At every point in \(X\) there is a neighborhood \(U\) such that for every geodesic triangle \(\Delta\subseteq U\) with comparison triangle \(\overline{\Delta}\subseteq M_k^2\),
\begin{equation}\label{curvature bound}
d(u,v) \leq d_k(\bar{u},\bar{v})
\end{equation}
for all \(u,v\in \Delta\) and their comparison points \(\bar{u},\bar{v} \in \overline{\Delta}\) (see Figure \ref{fig:comparison-triangles}).
Similarly, \(X\) is said to be of \emph{curvature bounded below} (or \emph{curvature
\(\geq k\)}) if \(d(u,v) \geq d_k(\bar{u},\bar{v})\).
In either case, the neighborhood \(U\) is referred to as a \emph{region of bounded curvature}.

\begin{figure}[h!] % FIGURE %
\begin{center}
\includegraphics[width=0.80\linewidth]{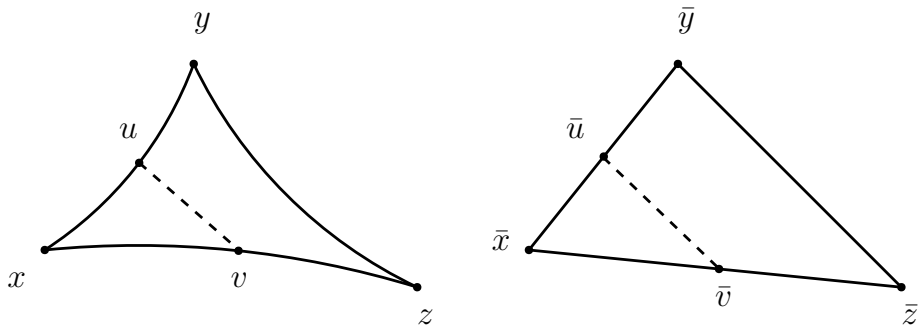}
\caption{A geodesic triangle with vertices \(x\), \(y\), and \(z\) in a length space (left) and the respective comparison triangle \(\overline{\Delta}(x,y,z)\) in \(M^2_0\) (right). The points \(\bar{u}\) and \(\bar{v}\) are chosen to satisfy \(d(x,u) = d_k(\bar{x},\bar{u})\) and \(d(x,v) = d_k(\bar{x},\bar{v})\). }
\label{fig:comparison-triangles}
\end{center}
\end{figure}

An \emph{Alexandrov space} is a complete and locally compact length space with curvature bounded either above or below. 
It should be noted that this definition of Alexandrov space (which comes from \cite{shiohama}) is not necessarily the uniformly accepted definition, but is necessary for our main theorem. In \cite{bbi2001}, an Alexandrov space is simply a length space with curvature bounded above or below. Due to the Hopf-Rinow Theorem, our additional requirement that the space be complete and locally compact allows us to avoid continually addressing the existence of shortest paths in the hypothesis of every proposition; however, it does limit the scope of some of the following preliminary results. In any case, a more general treatment of spaces of bounded curvature (in the sense of Alexandrov) can be found in \cite{bbi2001} or \cite{bridson}.

\subsection{Properties of Alexandrov Spaces} % PROPERTIES OF ALEXANDROV SPACES %

It is well known that for Alexandrov spaces, the angle between two geodesics emanating from a common point always exists.

\begin{lemma}\label{angle exists}
Let \(X\) be an Alexandrov space. If \(\gamma:[0,T]\to X\) and \(\eta:[0,S]\to X\) are shortest paths with \(\gamma(0) = \eta(0)\) then the angle \(\angle(\gamma,\eta)\) exists and
\[ \angle(\gamma,\eta) = \lim_{t\to 0^+} \angle^k (\gamma(t),\eta(t)) . \]
\end{lemma}

\begin{proof}
Suppose that \(X\) is of curvature \(\leq k\). Fix \(s\in (0,S]\) and \(a,b\in (0,T]\) such that \(a<b\). We will consider two distinct comparison triangles in \(M^2_k\). For simplicity of notation, we will denote them
\[ \overline{\Delta}(a) := \overline{\Delta}\big(\gamma(a),\gamma(0),\eta(s)\big) \quad \text{ and } \quad \overline{\Delta}(b) := \overline{\Delta}\big( \gamma(b),\gamma(0),\eta(s)\big) . \]

From the definition of \(\overline{\Delta}(a)\) we have (see Figure \ref{fig:monotoneangles})
\[
d_k(\overline{\gamma(a)}, \overline{\eta(s)})=d({\gamma(a)}, {\eta(s)}) .
\]
Let \(\tilde a\) be the comparison point of $\gamma(a)$ in \(\overline{\Delta}(b)\) (as opposed to \(\overline{\gamma(a)}\), which is the comparison point in \(\overline{\Delta}(a)\)). The upper bound $k$ for
the curvature gives
\[
d({\gamma(a)}, {\eta(s)}) \leq d_k(\tilde a, \overline{\eta(s)}) .
\]
Thus\footnote{Note that the left distance is in $\overline{\Delta}(a)$ while the
right distance is in $\overline{\Delta}(b)$.}
\(d_k \big(\overline{\gamma(a)} , \overline{\eta(s)} \big) \leq d_k\big(\tilde{a} ,
\overline{\eta(s)}\big) \), which in turn implies
\[
\angle^k \big(\gamma(a),\eta(s)\big) \leq \angle^k \big( \gamma(b) , \eta(s) \big) .
\]

\begin{figure}[h!] % FIGURE %
\begin{center}
\includegraphics[width=0.90\linewidth]{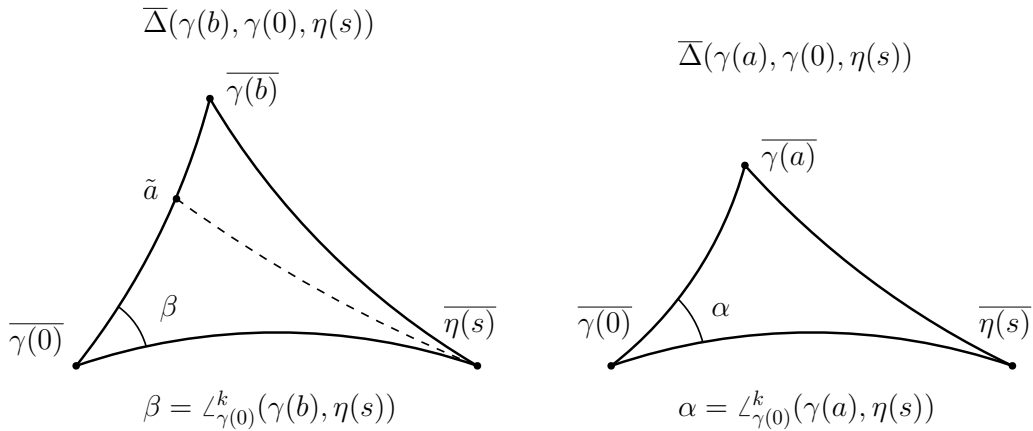}
\caption{An illustration of the comparison triangles \(\overline{\Delta}(\gamma(0),\gamma(b),\eta(s))\) and \(\overline{\Delta}(\gamma(0),\gamma(b),\eta(s))\) from Lemma \ref{angle exists}. }
\label{fig:monotoneangles}
\end{center}
\end{figure}

Thus, for any fixed \(s_0\in (0,S]\), the map \(t \mapsto \angle^k(\gamma(t),\eta(s_0))\) is
monotonically nondecreasing. By the same reasoning the map \(s\mapsto \angle^k(\gamma(t_0),\eta(s))\)
is nondecreasing for any fixed \(t_0\in (0,T]\). It follows from the monotonicity in both
coordinates\footnote{For clarification on monotonicity in functions of two variables, see Proposition \ref{prop: monotone two} of the Appendix.} that
\[ \angle^+(\gamma,\eta) = \limsup_{s,t\to0^+} \angle^k (\gamma(t),\eta(s)) = \liminf_{s,t\to0^+} \angle^k (\gamma(t),\eta(s)) = \angle^-(\gamma,\eta) . \]
We conclude that the angle \(\angle(\gamma,\eta)\) exists and is equal to \(\lim_{t\to 0^+}\angle^k(\gamma(t),\eta(t))\).

If \(X\) is of curvature \(\geq k\), the same method of proof applies, but the inequalities are reversed and the maps \(t\mapsto \angle^k(\gamma(t),\eta(s_0))\) and \(s\mapsto \angle^k(\gamma(t_0),\eta(s))\) are monotonically nonincreasing.
\end{proof}

\begin{corollary}\label{cor: angle condition}
Let \(X\) be an Alexandrov space of curvature \(\leq k\) (resp. \(\geq k\)). If the shortest paths \(\gamma:[0,T]\to X\) and \(\eta:[0,S]\to X\) (with \(\gamma(0) = \eta(0)\)) are contained in a region of bounded curvature, then
\[\angle(\gamma,\eta) \leq \angle^k(\gamma(t),\eta(s)) \quad \big(\text{ resp. } \angle(\gamma,\eta) \geq \angle^k(\gamma(t),\eta(s)) \ \big)\]
for any \(s,t>0\).
\end{corollary}

\begin{proof}
By Lemma \ref{angle exists}, if \(X\) is of curvature \(\leq k\) (resp. \(\geq k\)) the map \(t\mapsto \angle^k(\gamma(t),\eta(t))\) is nondecreasing (resp. nonincreasing). It follows immediately that \(\angle(\gamma,\eta)\leq \angle^k(\gamma(t),\eta(s))\) (resp. \(\angle(\gamma,\eta)\geq \angle^k(\gamma(t),\eta(s))\)) for any \(t\in (0,T]\) and \(s\in (0,S]\).
\end{proof}

While spaces of curvature bounded above and below share many properties, the following lemma gives an example of a property of spaces of curvature \(\geq k\) which is not valid in spaces of curvature \(\leq k\). This lemma also makes use of notation we shall need again, so we introduce it here. Let \(\gamma:[0,T]\to X\) be a path and fix \(t\in (0,T)\). The path \(\gamma|_{[t,0]}\) is defined by \(\gamma|_{[t,0]}(s) = \gamma(t-s)\) for \(s\in [0,t]\). In other words, \(\gamma|_{[t,0]}\) is the path that runs backwards along \(\gamma\) from \(\gamma(t)\) to \(\gamma(0)\).

\begin{lemma}\label{supplementary angles}
If \(X\) is an Alexandrov space of curvature bounded below, \(\gamma: [0,T]\to X\) is a shortest path, \(0 < t < T\), and \(\sigma_t:[0,S] \to X\) is a shortest path with \(\sigma_t(0) = \gamma(t)\) then
\[ \angle_{\gamma(t)} \big( \gamma |_{[t,T]} , \sigma_t \big) + \angle_{\gamma(t)} \big( \gamma |_{[t,0]} , \sigma_t \big) = \pi . \]
In other words, adjacent angles along a shortest path sum to \(\pi\).
\end{lemma}

\begin{proof}
By Proposition \ref{triangle inequality angles}, we know that
\[\angle_{\gamma(t)} \big( \gamma |_{[t,T]} , \sigma_t \big) + \angle_{\gamma(t)} \big( \gamma |_{[t,0]} , \sigma_t \big) \geq \angle_{\gamma(t)} \big( \gamma|_{[t,T]}, \gamma|_{[t,0]} \big) = \pi \]
so it suffices to prove the reverse inequality.

Fix a small \(\delta>0\). We will consider a configuration of comparison points in \(M^2_k\) for the points \(\gamma(t-\delta)\), \(\gamma(t)\), \(\gamma(t+\delta)\), and \(\sigma_t(\delta)\).

First, consider the comparison triangle \(\overline{\Delta}(\gamma(t-\delta),\sigma_t(\delta),\gamma(t+\delta)\) with the comparison point \(\overline{\gamma(t)}\). Second, consider the comparison triangle \(\overline{\Delta}(\gamma(t),\sigma_t(\delta),\gamma(t+\delta))\) with the points \(\overline{\gamma(t)}\) and \(\overline{\gamma(t+\delta)}\) aligned as in Figure \ref{fig:supplementary}. Given that each triangle has a vertex representing \(\sigma_t(\delta)\), we have labeled them \(\overline{\sigma_t(\delta)}\) and \(\widehat{\sigma_t(\delta)}\) respectively, to distinguish them.

\begin{figure}[h!] % FIGURE %
\begin{center}
\includegraphics[width=0.6\linewidth]{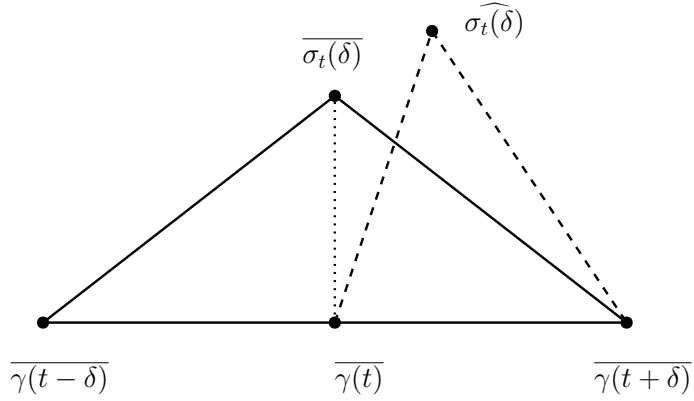}
\caption[width=0.5\linewidth]{The comparison point construction of Lemma \ref{supplementary angles}. }
\label{fig:supplementary}
\end{center}
\end{figure}

By the definition of curvature \(\geq k\), we know that
\[d_k\big( \overline{\sigma_t(\delta)},\overline{\gamma(t)} \big) \leq d_k \big( \widehat{\sigma_t(\delta)},\overline{\gamma(t)} \big) = d\big( \sigma_t(\delta),\gamma(t) \big) . \]
Considering that
\begin{align*} 
& d_k\big( \overline{\sigma_t(\delta)},\overline{\gamma(t+\delta)} \big) = d_k \big( \widehat{\sigma_t(\delta)},\overline{\gamma(t+\delta)} \big) \\ 
\text{and } & d_k\big( \overline{\gamma(t)},\overline{\gamma(t+\delta)} \big) = d_k \big( \overline{\gamma(t)},\overline{\gamma(t+\delta)} \big) , 
\end{align*}
we have an inequality between the interior angles at \(\overline{\gamma(t)}\);
\[ \angle^k_{\gamma(t)} \big(\sigma_t(\delta),\gamma(t+\delta) \big) = \angle^k_{\overline{\gamma(t)}} \big(\widehat{\sigma_t(\delta)},\overline{\gamma(t+\delta)} \big) \leq \angle^k_{\overline{\gamma(t)}} \big( \overline{\sigma_t(\delta)},\overline{\gamma(t+\delta)} \big) . \]
Applying the analogous argument to \(\overline{\Delta}(\gamma(t-\delta),\sigma_t(\delta),\gamma(t))\), we see that
\[\angle^k_{\gamma(t)} \big(\sigma_t(\delta),\gamma(t-\delta) \big) + \angle^k_{\gamma(t)} \big(\sigma_t(\delta),\gamma(t+\delta) \big) \leq \angle^k_{\gamma(t)} \big(\gamma(t-\delta),\gamma(t+\delta) \big) . \]
Taking the limit as \(\delta\to 0^+\) yields the result.
\end{proof}

\begin{proposition}[Semi-continuity of angles]\label{prop: semi-continuity}
Let \(X\) be an Alexandrov space of curvature bounded above (resp. below). Suppose that the sequences of shortest paths \(\{\gamma_n\}_{n=1}^\infty\) and \(\{\sigma_n\}_{n=1}^\infty\), with \(\gamma_n(0) = \sigma_n(0)\) for all \(n\), converge uniformly to shortest paths \(\gamma\) and \(\sigma\) respectively. Then \(\angle (\gamma,\sigma) \geq \limsup_{n\to\infty} \angle (\gamma_n,\sigma_n)\) (resp. \(\angle(\gamma,\sigma) \leq \liminf_{n\to\infty} \angle (\gamma_n,\sigma_n)\)).
\end{proposition}

\begin{proof}
First, suppose that \(X\) is of curvature \(\leq k\). For any \(t\in [0,T]\), since \(\gamma_n \to \gamma\) uniformly, \(\gamma_n(t) \to \gamma(t)\); and the same can be said for the path \(\sigma\). By Lemma \ref{angle exists} and Corollary \ref{cor: angle condition},
\begin{align}
\angle (\gamma,\sigma) & = \lim_{t\to 0^+} \angle^k_{\gamma(0)} \big( \gamma(t) , \sigma(t) \big) 			 \label{line1}
\\ & = \lim_{t\to 0^+} \bigg( \lim_{n\to\infty} \angle^k_{\gamma_n(0)} \big( \gamma_n(t) , \sigma_n(t) \big) \bigg) 	 \label{line2}
\\ & \geq \lim_{t\to 0^+} \bigg( \limsup_{n\to\infty} \angle ( \gamma_n, \sigma_n) \bigg) . 					 \label{line3}
\end{align}
As the final quantity above is independent of \(t\), we have \(\angle (\gamma,\sigma) \geq \limsup_{n\to\infty} \angle (\gamma_n,\sigma_n)\).

Alternatively, if we suppose that \(X\) is of curvature \(\geq k\). Then (\ref{line1}) and (\ref{line2}) above still hold, but in (\ref{line3}) we make use of the other inequality of Corollary \ref{cor: angle condition} to obtain \(\angle (\gamma,\sigma) \leq \liminf_{n\to\infty} \angle (\gamma_n,\sigma_n)\).
\end{proof}

\section{Right-Derivative of Distance to a Compact Set}\label{sec: derivative} %%% DISTANCE TO A COMPACT SET %%%

In any metric space, it is an easy application of the triangle inequality to show that distance to a set is 1-Lipschitz. 
It follows that distance along a geodesic is differentiable almost everywhere in the domain of the geodesic. In this section, we prove an explicit value for this (one-sided) derivative, when the distance is taken to a compact set. 
Even more than finding the derivative of distance along a geodesic, if we think of the geodesic \(\gamma\) as representing a direction in an Alexandrov space, this formula is akin to the directional derivative of distance in the direction \(\gamma\). This idea serves as a precursor to developing gradients of functions on Alexandrov spaces, which is explored in more detail in \cite[Section 7.4]{plaut}, and in a more general setting of `geometric' metric spaces beyond Alexandrov spaces in \cite{lytchak}.

\begin{lemma}\label{limsup}
If \(X\) is an Alexandrov space, \(\gamma:[0,T]\to X\) is a shortest path, and \(p\) is an element of \(X\) such that \(\gamma(0)\neq p\), then
\[\limsup_{t\to 0^+} \frac{ d(\gamma(t),p) - d(\gamma(0),p)}{t} \leq - \cos \big( \angle_{\min} \big) \]
where \(\angle_{\min}\) is the infimum of angles between \(\gamma\) and shortest paths from \(\gamma(0)\) to \(p\).
\end{lemma}

\begin{proof}
Let \(\eta :[0,S]\to X\) be a shortest path connecting \(\gamma(0)\) to \(p\). Using
Lemma \ref{lem:strong angle} and the fact that \(-\cos\) is nondecreasing on \([0,\pi]\),
\begin{align*} 
-\cos \big( \angle (\gamma,\eta) \big) & \geq \limsup_{t \to 0^+} \bigg( - \cos\big( \angle^k(\gamma(t),p)\big) \bigg) 
\\ & = \limsup_{t\to 0^+} \frac{- d(\gamma(0),p) + d(\gamma(t),p)}{t} , 
\end{align*}
where the last equality comes from Lemma \ref{lem: cosine inequality}. Therefore,
\[ \limsup_{t\to 0^+} \frac{d(\gamma(t),p) - d(\gamma(0),p)}{t} \leq -\cos \big( \angle (\gamma,\eta) \big) . \]
Given that this holds for any shortest path \(\eta\) connecting \(\gamma(0)\) to \(p\), we can replace \(\angle (\gamma,\eta)\) above with \(\angle_{\min}\).
\end{proof}

\begin{lemma}\label{lem:pi angle}
Let \(X\) be an Alexandrov space, \(K\) a compact set in \(X\), and \(\gamma:[0,T] \to X\) a shortest path such that \(\gamma(0)\notin K\). For each \(t\in [0,T]\) let \(\sigma_t\) be a shortest path connecting \(\gamma(t)\) to \(K\). If there is a sequence \(\{t_n\}_{n=1}^\infty\) such that \(t_n\to 0\) and the sequence of shortest paths \(\{\sigma_{t_n}\}_{n=1}^\infty\) converges to \(\sigma_0\), then
\[\limsup_{n\to\infty} \angle^k_{\gamma(t_n)} \big( \gamma(0),\sigma_{t_n}(s) \big) \leq \pi - \angle(\gamma,\sigma_0) \]
for all sufficiently small \(s>0\). (See Figure \ref{fig:firstvar}.)
\end{lemma}

\begin{proof}
Begin by fixing \(N\in \N\) and \(s'\in (0,S]\) such that \(\gamma(0)\), \(\gamma(t_n)\), and \(\sigma_{t_n}(s)\) all lie in a region of bounded curvature whenever \(n\geq N\) and \(s\leq s'\). For simplicity of notation let \(s_n = \sigma_{t_n}(s)\) for some fixed \(s>0\).

Suppose \(X\) is of curvature \(\geq k\). Then
\begin{align*}
\limsup_{n\to\infty} \angle^k_{\gamma(t_n)} (\gamma(0),s_n) & \leq \limsup_{n\to\infty} \angle_{\gamma(t_n)} (\gamma|_{[t_n,0]}, \sigma_{t_n})
\\ & = \pi - \liminf_{n\to\infty} \angle_{\gamma(t_n)}(\gamma|_{[t_n,T]},\sigma_{t_n})
\\ &  \leq \pi - \angle (\gamma,\sigma_0)
\end{align*}
by Corollary \ref{cor: angle condition}, Lemma \ref{supplementary angles}, and Proposition \ref{prop: semi-continuity}, respectively.

\begin{figure}[h!] % FIGURE %
\begin{center}
\includegraphics[width=0.8\linewidth]{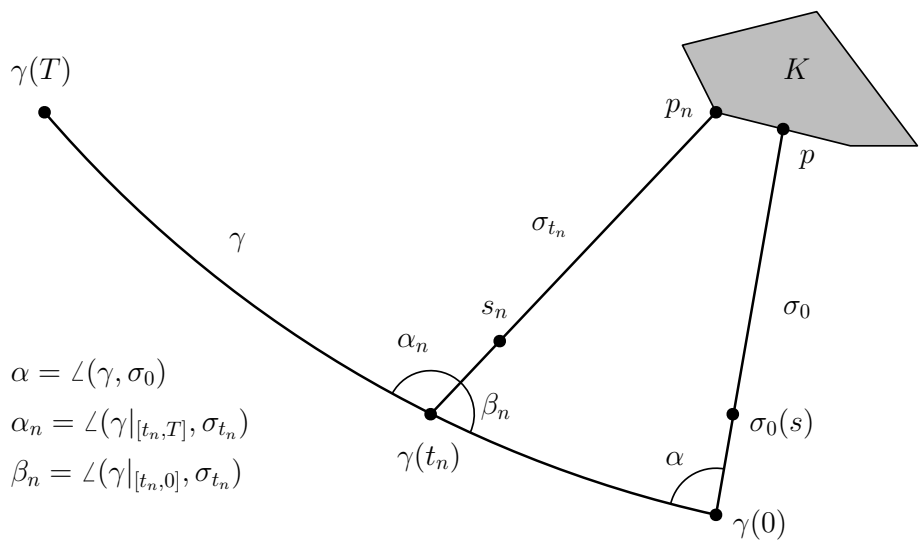}
\caption[width=0.5\linewidth]{An illustration of the paths, points, and angles in the proof of Lemma \ref{lem:pi angle} and Theorem \ref{thm:compact set derivative}. }
\label{fig:firstvar}
\end{center}
\end{figure}

Next, suppose \(X\) is of curvature \(\leq k\).
For each \(n\), let \(\eta_n\) be a shortest path connecting \(\gamma(0)\) to \(s_n\) (see Figure
\ref{fig:firstvar2}). By Corollary \ref{cor: angle condition}, \(\angle(\eta_n,\sigma_0) \leq \angle^k_{\gamma(0)} (s_n,\sigma_0(s))\) for all \(n\). Since \(\sigma_{t_n} \to \sigma_0\), we have \(\angle^k_{\gamma(0)}(s_n,\sigma_0(s)) \to 0\) and so \(\angle(\eta_n,\sigma_0) \to 0\).
By Proposition \ref{triangle inequality angles} twice,
\[ \angle (\gamma,\sigma_0) \leq \angle (\gamma,\eta_n) + \angle (\eta_n,\sigma_0) \leq \angle (\gamma,\sigma_0)  + 2 \angle (\eta_n,\sigma_0) . \]
So, as \(\angle (\eta_n,\sigma_0) \to 0\), we have \(\angle (\gamma,\eta_n) \to \angle (\gamma,\sigma_0)\).
Thus, using Corollary \ref{cor: angle condition} again
\begin{equation}
\angle (\gamma,\sigma_0) = \liminf_{n\to\infty} \angle (\gamma,\eta_n) \leq
\liminf_{n\to\infty} \angle^k_{\gamma(0)} (\gamma(t_n),s_n) .
\label{eq:inequality}
\end{equation}

\begin{figure}[h!] % FIGURE %
\begin{center}
\includegraphics[width=0.60\linewidth]{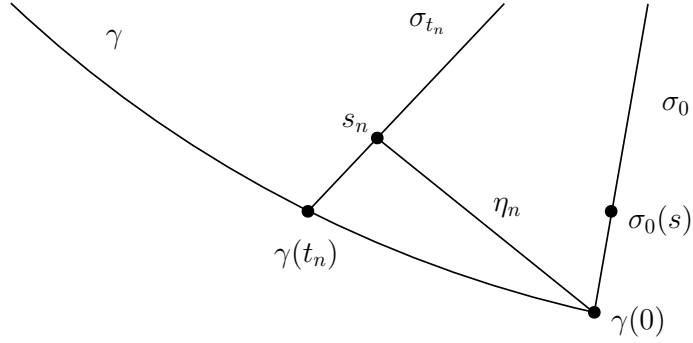}
\caption[width=0.5\linewidth]{The relationship between the paths \(\gamma\), \(\eta_n\), and \(\sigma_0\).}
\label{fig:firstvar2}
\end{center}
\end{figure}

Now let \(\delta >0\) be given. By the fact that infinitesimal triangles are Euclidean, we know
that in the comparison triangle \(\overline{\Delta}(\gamma(t_n),s_n,\gamma(0))\), we have for $n$ large enough
\begin{equation}\label{sum angles}
\angle^k_{\gamma(0)}(\gamma(t_n) , s_n) + \angle^k_{\gamma(t_n)}(\gamma(0) , s_n) + \angle^k_{s_n}(\gamma(t_n) , \gamma(0)) < \pi + \delta .
\end{equation}

Since \(\angle^k_{s_n} (\gamma(t_n),\gamma(0)) \to 0\), with (\ref{eq:inequality}) this gives
\[ \limsup_{n\to\infty}  \angle^k_{\gamma(t_n)}(\gamma(0) , s_n) < \pi + \delta - \liminf_{n\to\infty} \angle^k_{\gamma(0)}(\gamma(t_n),s_n) \leq \pi + \delta - \angle(\gamma,\sigma_0)\]
Letting \(\delta>0\) go to zero gives the desired result.
\end{proof}

\begin{theorem}\label{thm:compact set derivative}
Let \(X\) be an Alexandrov space, \(\gamma:[0,T] \to X\) a shortest path, and \(K\) a compact set not containing \(\gamma(0)\). If \(\ell(t) = d(\gamma(t), K)\), then
\[\lim_{t\to 0^+} \frac{\ell(t) - \ell(0)}{t} = -\cos(\angle_{\min})\]
where \(\angle_{\min}\) is the infimum of angles between \(\gamma\) and any shortest path of length \(\ell(0)\) which connects \(\gamma(0)\) to \(K\).
\end{theorem}

\begin{proof}
First, let \(\eta_0\) be a shortest path connecting \(\gamma(0)\) to \(K\) and let \(a\in K\) be the endpoint of \(\eta_0\). Note that for each \(t>0\), \(\ell(t) \leq d(\gamma(t),a)\). Therefore, by Lemma \ref{limsup},
\[\limsup_{t\to 0^+} \frac{\ell(t) - \ell(0)}{t} \leq \limsup_{t\to 0^+} \frac{ d(\gamma(t),a) - d(\gamma(0),a)}{t} \leq - \cos \big( \angle_{\min} \big) . \]

To get the reverse estimate, let \(\{t_n\}_{n=1}^\infty\) be a sequence in \((0,T]\) such that \(t_n \to 0\) and
\[ \lim_{n\to\infty} \frac{\ell(t_n) - \ell(0)}{t_n} = \liminf_{t\to 0^+} \frac{\ell(t) - \ell(0)}{t} . \]
Similar to Lemma \ref{lem:pi angle}, for each \(n\) let \(\sigma_{t_n}\) be a shortest path connecting \(\gamma(t_n)\) to \(K\). Since \(K\) is compact, the length of each path in the sequence \(\{\sigma_{t_n}\}_{n=1}^\infty\) is uniformly bounded. Therefore, by the Arzela-Ascoli Theorem,\footnote{see Appendix \ref{sec: arzela-ascoli} for clarification on how Arzela-Ascoli is used here.} \(\{\sigma_{t_n}\}_{n=1}^\infty\) contains a subsequence which converges uniformly to a shortest path \(\sigma_0\) connecting \(\gamma(0)\) to \(K\). Without loss of generality, we assume that the sequence \(\{\sigma_{t_n}\}_{n=1}^\infty\) is this uniformly convergent subsequence.

Fix \(s\) sufficiently small to satisfy the hypothesis of Lemma \ref{lem:pi angle}. For simplicity of notation let \(p_n\in K\) be the endpoint of \(\sigma_{t_n}\), let \(p\in K\) be the endpoint of \(\sigma_0\), and let \(s_n = \sigma_{t_n}(s)\) (see Figures \ref{fig:firstvar} and \ref{fig:firstvar2}).
By Lemma \ref{lem: cosine inequality}, 
\begin{equation}\label{eq:cosine}
\liminf_{n\to\infty} \frac{s-d(s_n,\gamma(0))}{t_n} =  \liminf_{n\to\infty} \cos\big( \angle^k_{\gamma(t_n)} (\gamma(0),s_n) \big) .
\end{equation}
Note that \(\ell(t_n) = s + d(s_n,p_n)\) and
\[\ell(0) \leq d(\gamma(0),p_n) \leq d(\gamma(0),s_n) + d(s_n,p_n) .\]
Combining these observations with (\ref{eq:cosine}), we get
\begin{align*} 
\liminf_{n\to\infty} \frac{\ell(t_n) - \ell(0)}{t_n} & \geq \liminf_{n\to\infty} \cos\big( \angle^k_{\gamma(t_n)} (\gamma(0),s_n) \big) 
\\ & = \cos \big( \limsup_{n\to\infty} \angle^k_{\gamma(t_n)} (\gamma(0),s_n) \big) . 
\end{align*}
Then by Lemma \ref{lem:pi angle},
\[ \cos \big( \limsup_{n\to\infty} \angle^k_{\gamma(t_n)} (\gamma(0),s_n) \big) \geq \cos \big( \pi - \angle(\gamma,\sigma_0) \big) = - \cos \big( \angle(\gamma,\sigma_0)\big) .\]
Thus,
\begin{align*} 
\liminf_{t\to 0^+} \frac{\ell(t) - \ell(0)}{t} & = \liminf_{n\to\infty} \frac{\ell(t_n) - \ell(0)}{t_n} 
\\ & \geq - \cos \big( \angle(\gamma,\sigma_0)\big) 
\\ & \geq -\cos(\angle_{\min}) 
\end{align*}
which is the desired reverse estimate.
\end{proof}

\appendix %%%%%%% APPENDIX %%%%%%%

\section{Counterexamples} 

\subsection{Upper Angle \(\neq\) Lower Angle}\label{sec: upper lower}

Consider \(\R^2\) with the metric \(d(x,y) = |x_1 - y_1|+|x_2 - y_2|\), known as the \(\ell^1\) metric, or the `taxicab' metric. The space \((\R^2,d)\) is a complete and locally compact length space.

Let \(\gamma:[0,1]\to \R^2\) be defined by \(\gamma(t) = (t,t)\) and let \(\eta:[0,1]\to \R^2\) be defined by \(\eta(s) = (s,0)\). Note that these are both shortest paths with
\[ L(\gamma) = 2 = d\big( (0,0), (1,1) \big) \quad \text{and} \quad L(\eta) = 1 = d\big( (0,0) , (1,0) \big) \]
Since \(\gamma(0) = \eta(0)\), we may consider the upper and lower angle between them.

\begin{figure}[h!] % FIGURE %
\begin{center}
\includegraphics[width=0.60\linewidth]{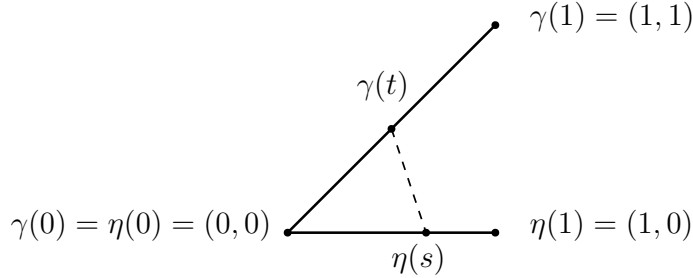}
\caption[width=0.5\linewidth]{The paths \(\gamma\) and \(\eta\). }
\label{fig:L1angle}
\end{center}
\end{figure}

First, let \(s=t\). Then (see Figure \ref{fig:L1angle}),
\[ d\big( \gamma(0),\gamma(t) \big) = 2t \quad \text{and} \quad
d\big( \eta(0),\gamma(t) \big) = d\big( \gamma(t) , \eta(t) \big) = t . \]
Therefore, \(\angle^0 (\gamma(t),\eta(t)) = 0\) for all \(t\). It follows that
\[ \angle^-(\gamma,\eta) = \liminf_{s,t\to 0^+} \angle^0 (\gamma(t),\eta(s)) \leq \lim_{t\to 0^+} \angle^0 (\gamma(t),\eta(t)) = 0 \]
Furthermore, as \(0\) is the minimum possible angle, we have \(\angle^-(\gamma,\eta) = 0\).

Next, consider \(t=s^2\). We have 
\begin{align*} 
& d\big( \gamma(0),\gamma(s^2) \big) = 2s^2  
\\ & d\big( \eta(0),\eta(s) \big) = s 
\\ \text{and } & d\big( \gamma(s^2) , \eta(s) \big) = s-s^2+s^2=s . 
\end{align*}
Note that the (Euclidean) comparison triangle \(\overline{\Delta}(\gamma(0),\gamma(t),\eta(s))\)
is isosceles with a very small base. Using elementary plane geometry, one easily derives that
\( \cos\big(\angle^0(\gamma(s^2),\eta(s)) \big) = s.\)
Since cosine is continuous and nonincreasing on the interval \([0,\pi]\), we see that
\[
\cos\big( \angle^+ (\gamma,\eta) \big)  \leq \lim_{s \to 0^+}
\cos\big( \angle^0(\gamma(s^2),\eta(s)) \big) = \lim_{s \to 0^+} s = 0 .
\]
Thus we have \(\angle^+ (\gamma,\eta) \geq \pi/2\). In fact, we can show that the upper angle equals \(\pi/2\).

If the upper angle were greater than \(\pi/2\), then there would have to be points \(s,t\in [0,1]\) such that \(d(\gamma(t),\eta(s))\) is greater than both \(d(\gamma(0),\gamma(t))\) and \(d(\eta(0),\eta(s))\).
However, this is impossible since \(d(\eta(0),\eta(s)) = s\), \(d(\gamma(0),\gamma(t)) = 2t\), and
\[d(\gamma(t),\eta(s)) = |s-t| + |t-0| = t+|s-t| \]
which cannot be simultaneously greater than \(s\) and \(2t\) for any \(s,t\in [0,1]\).

\subsection{Unbounded Curvature}

Consider again the length space \((\R^2,d)\) from Appendix \ref{sec: upper lower}.
While we could use the result of the previous section combined with Lemma \ref{angle exists} to establish that \((\R^2,d)\) does not have bounded curvature, we instead provide here a direct proof using the definition of bounded curvature given in Section \ref{alexandrov spaces}.

Given any neighborhood \(U\) in \(\R^2\), and any \(k\in \R\), we can find three points \(x,y,z\in U\) such that
\[d(x,y) = d(x,z) = d(y,z) \]
and \(d(x,y)<D_k/2\), which forms an equilateral comparison triangle \(\overline{\Delta}(x,y,z)\subseteq M^2_k\).

\begin{figure}[h!] % FIGURE %
\begin{center}
\includegraphics[width=0.80\linewidth]{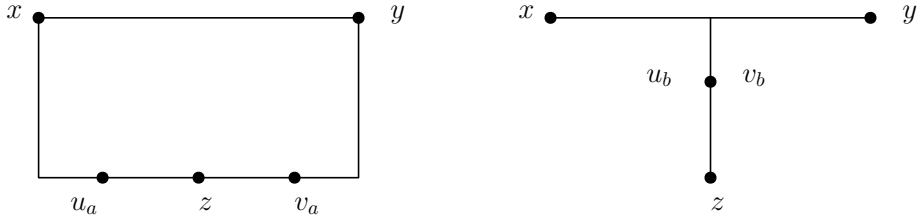}
\caption[width=0.5\linewidth]{The geodesic triangles \(A\) (left) and \(B\) (right) with vertices \(x\), \(y\), and \(z\) in the `taxicab' space \((\R^2 , d)\). }
\label{fig:twotriangles}
\end{center}
\end{figure}

While there are many geodesic triangles in \((\R^2,d)\) with vertices \(x\), \(y\), and \(z\), we will only consider two. First, we choose the shortest paths which form a rectangle around the points and call this triangle \(A\) (left side of Figure \ref{fig:twotriangles}).
We may fix points \(u_a, v_a\in A\) on each side of the point \(z\) such that \(d(u_a,v_a) = d(z,u_a) + d(z,v_a)\). However, in our comparison triangle \(\overline{\Delta}(x,y,z)\), we have
\[d_k(\bar{u}_a,\bar{v}_a)< d_k(\bar{u}_a,\bar{z}) + d_k(\bar{z},\bar{v}_a) = d(u_a,z)+d(z,v_a) = d(u_a,v_a) . \]
Since \(d(u_a,v_a) > d_k(\bar{u}_a,\bar{v}_a)\), the curvature of \((\R^2,d)\) cannot be \(\leq k\).

Second, we consider the geodesic triangle \(B\) consisting of three branching geodesics (right side of Figure \ref{fig:twotriangles}).
Let \(u_b,v_b\in B\) be such that \(d(u_b,z) = d(v_b,z) > 0\) and \(d(u_b,v_b) = 0\). We know that the points \(u_b\) and \(v_b\) exist since the shortest paths connecting \(z\) to \(x\) and \(z\) to \(y\) are branches from a common geodesic. Recalling that our comparison triangle \(\overline{\Delta}(x,y,z)\) is equilateral, we have
\[ d_k(\bar{u}_b,\bar{v_b}) > 0 = d(u_b,v_b) \]
so the curvature of \((\R^2,d)\) cannot be \(\geq k\).
As no neighborhood \(U\) can satisfy the definition of curvature bounded above or below by any \(k\), the space \((\R^2,d)\) is not of bounded curvature.

\subsection{Supplementary Upper Angles May Not Sum to \(\pi\)}\label{sec:triangle-inequ}

If on a geodesic $\gamma:[-T,T]\mapsto X$ where $X$ is Alexandrov with lower bound on the curvature,
we choose 3 nearby points $b=\gamma(-t)$, $a=\gamma(0)$, and $c=\gamma(t)$, then the upper angle
\(\angle^+_{\gamma(0)} \big( \gamma |_{[0,t]} , \gamma |_{[0,-t]}\big)\) equals $\pi$. This follows
immediately from the definition of angles. See Figure \ref{fig:two-angles}.

\begin{figure}[h!] % FIGURE %
\begin{center}
\includegraphics[width=0.60\linewidth]{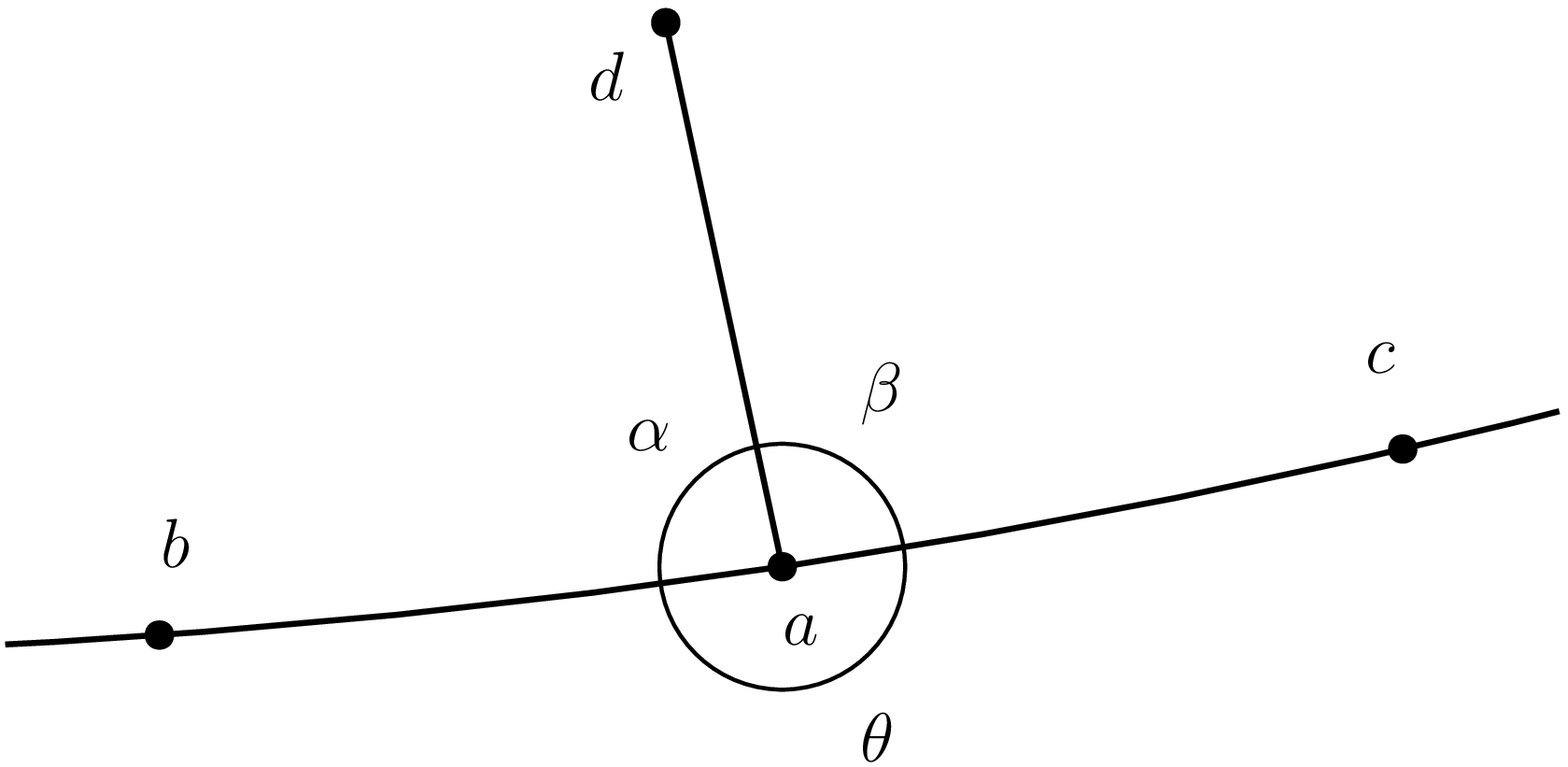}
\caption[width=0.5\linewidth]{ }
\label{fig:two-angles}
\end{center}
\end{figure}

However, spaces without lower bound on the curvature, such as the space \((\R^2,d)\) from Appendix
\ref{sec: upper lower}, have the property that two geodesics that agree on a segment may bifurcate.
Thus, suppose in Figure \ref{fig:two-angles}, $bad$ and $dac$ are geodesics.
By the previous observation all three angles $\alpha$, $\beta$, and $\theta$ are equal to $\pi$,
and we have a counter example to Lemma \ref{supplementary angles}. Notice that it also follows that
if there is a lower bound on the curvature, then geodesics cannot bifurcate.

\section{Results in Analysis}

\subsection{Monotonicity in Functions of Two Real Variables}

Let \(f:\R^2 \to \R\) be a function. We say that \(f\) is \textbf{component-wise monotonic} if for any constant \(r\in \R\), the maps \(x\mapsto f(x,r)\) and \(x\mapsto f(r,x)\) are both nondecreasing or both nonincreasing.
We further define the right-sided limit superior of \(f\) at \(a\) as
\[ \limsup_{x,y\to a^+} f(x,y) = \lim_{\varepsilon\to 0} \sup \{ f(x,y) : a < x,y \leq a+\varepsilon \} . \]
Similarly, the right-sided limit inferior of \(f\) at \(a\) is given by
\[ \liminf_{x,y\to a^+} f(x,y) = \lim_{\varepsilon\to 0} \inf \{ f(x,y) : a < x,y \leq a+\varepsilon \} . \]

\begin{lemma}\label{prop: monotone two}
If \(f:\R^2\to\R\) is component-wise monotonic, then for any \(a\in \R\),
\[\limsup_{x,y\to a^+} f(x,y) = \liminf_{x,y\to a^+} f(x,y) . \]
\end{lemma}

\begin{proof}
Let \(f: \R^2 \to \R\) be given and assume that \(f\) is component-wise nondecreasing (the result for nonincreasing follows by symmetry). If we fix \(x_1<x_2\), then by our assumptions,
\[ f(x_1,x_1) \leq f(x_1,x_2) \leq f(x_2 , x_2) . \]
Therefore, for any \(a\in \R\), it follows from the monotonicity of \(x\mapsto f(x,x)\) that
\[ \limsup_{x\to a^+} f(x,x) = \liminf_{x\to a^+} f(x,x) = \lim_{x\to a^+} f(x,x) . \]
Denote the limit by \(f(a)^+\). To finish the proof, it suffices to show that
\[ \limsup_{x,y\to a^+} f(x,y) \leq f(a)^+ \leq \liminf_{x,y\to a^+} f(x,y). \]

Let \(\{x_n\}\) and \(\{y_n\}\) be sequences such that \(x_n,y_n>a\) for all \(n\) and \(x_n,y_n\to a\). For each \(n\), define \(m_n = \min\{x_n,y_n\}\) and \(M_n = \max\{x_n,y_n\}\). Then
\begin{align*} 
\limsup_{n\to\infty} f(x_n,y_n) & \leq \lim_{n\to\infty} f(M_n,M_n) 
\\ & = f(a)^+ 
\\ & = \lim_{n\to\infty} f(m_n, m_n) 
\\ & \leq \liminf_{n\to \infty} f(x_n,y_n) . \qedhere 
\end{align*}
\end{proof}

\subsection{The Arzela-Ascoli Theorem for Paths}\label{sec: arzela-ascoli}

While there are many equivalent statements of the Arzela-Ascoli Theorem, the version which best fits our needs is that found in \cite{bridson}.

\begin{theorem}[Arzela-Ascoli]\label{arzela-ascoli}
If \(X\) is a compact metric space and \(Y\) is a separable metric space, then every sequence of equicontinuous maps \(f_n : Y\to X\) contains a uniformly convergent subsequence.
\end{theorem}

\begin{corollary}
If \(X\) is a compact metric space and \(\{\gamma_n\}_{n=1}^\infty\) is a sequence of paths in \(X\) with uniformly bounded lengths, then the sequence \(\{\gamma_n\}_{n=1}^\infty\) contains a uniformly convergent subsequence.
\end{corollary}

\begin{proof}
Without loss of generality, we may assume that each \(\gamma_n\) is constant-speed with domain \([0,1]\). Since the length of the paths is uniformly bounded by, say \(M\in \R\), for all \(t,t' \in [0,1]\)
\[ d\big( \gamma_n(t), \gamma_n(t') \big) \leq M |t - t'| \]
so \(\{\gamma_n\}_{n=1}^\infty\) is equicontinuous. As \([0,1]\) is separable and \(X\) is compact, by Theorem \ref{arzela-ascoli} \(\{\gamma_n\}_{n=1}^\infty\) contains a uniformly convergent subsequence.
\end{proof}

\section{A Geometric Observation}

\subsection{Space Forms are Infinitesimally Euclidean}\label{sec: inf-eucl}

\begin{lemma} 
The upper angle is independent of the curvature of \(M^2_k\). 
\end{lemma}

\begin{proof} There are many ways of proving this. The first is by using the geodesic
equation to establish that the smaller the domain, the closer a geodesic crossing it resembles
a "straight line". This is made rigorous in \cite{BV} (Proposition 1.10).

Another more informal way is to realize that instead of shrinking a triangle with sides of lengths
$a$, $b$, and $c$ in $M_k^2$ by a factor of, say $\varepsilon$, we might equally well define a new space $M$
by changing coordinates
in $M_k^2$ from $x$ to $\bar x=x/\varepsilon$. Now consider a triangle with those same
lengths $a$, $b$, and $c$ again. Gauss' original definition of (Gaussian) curvature is
\[
k=\lim_{A\rightarrow 0} \dfrac{A'}{A}
\]
where $A$ is the area of a small disk in $M_k^2$ and $A'$ is the area in the unit sphere
swept out by the unit normals in $A$. Clearly, in our rescaled space $A'$ has been shrunk
by a factor $\varepsilon^2$ while $A$ changes very little.
\end{proof}

\bibliography{metricgeometry}
\bibliographystyle{alpha}

\end{document}